\documentclass[12pt]{article}
\usepackage{amsmath}
\usepackage{amssymb}

\usepackage{mathtools}
\usepackage{amsthm}
\usepackage[usenames]{color}

\usepackage{graphicx}

\newcommand{\1}[1]{{\mathbf 1}{\{#1\}}}

\newcommand{\eps}{\varepsilon}
\newcommand{\Z}{{\mathbb Z}}

\newcommand{\V}{{\mathcal V}}
\newcommand{\RR}{{\mathcal R}}

\newcommand{\R}{{\mathbb R}}
\newcommand{\Ex}{{\mathcal{E}\!\!\text{\footnotesize{\textsl{x}}}}}

\newcommand{\s}{{\widehat S}}
\newcommand{\hN}{{\widehat N}}

\newcommand{\FFF}{{\mathcal F}}

\let\phi=\varphi

\newcommand{\cov}{{\mathop{\mathrm{cov}}}}
\newcommand{\Var}{{\mathop{\mathrm{Var}}}}
\newcommand{\8}{{\infty}}

\newcommand{\hh}{{\bar h}}
\newcommand{\hq}{{\bar q}}

\newcommand{\IP}{{\mathbb P}}
\newcommand{\IE}{{\mathbb E}}

\newcommand{\hP}{\widehat{P}}

\newcommand{\hH}{{\widehat H}}

\newcommand{\htau}{\widehat{\tau}}

\newcommand{\B}{{\mathsf B}}

\allowdisplaybreaks

\newtheorem{theo}{Theorem}[section]
\newtheorem{lem}[theo]{Lemma}
\newtheorem{df}[theo]{Definition}
\newtheorem{prop}[theo]{Proposition}
\newtheorem{cor}[theo]{Corollary}
\newtheorem{rem}[theo]{Remark}

\title{On the range of a two-dimensional conditioned 
simple random walk}
\author{Nina Gantert$^{1}$ \and
 Serguei~Popov$^{2}$ \and Marina Vachkovskaia$^{2}$}

\begin{document}

\maketitle

{\footnotesize 
\noindent $^{~1}$Technische Universit\"at M\"unchen, 
Fakult\"at f\"ur Mathematik,
Boltzmannstr.\ 3, 85748 Garching,
Germany\\
\noindent e-mail: \texttt{gantert@ma.tum.de}\\

\noindent $^{~2}$Department of Statistics, Institute of Mathematics,
 Statistics and Scientific Computation, University of Campinas --
UNICAMP, rua S\'ergio Buarque de Holanda 651,
13083--859, Campinas SP, Brazil\\
\noindent e-mails: \texttt{\{popov,marinav\}@ime.unicamp.br}

}

\begin{abstract}
We consider the two-dimensional simple random walk
conditioned on never hitting the origin. 
This process is a Markov chain, namely
it is the Doob $h$-transform of the simple random walk
with respect to the potential kernel. 
It is known to be transient and we show that it is 
``almost recurrent'' in the sense that each infinite set 
is visited infinitely often, almost surely.
We prove that, for a ``large'' set, the proportion
of its sites visited by the conditioned walk is 
approximately a Uniform$[0,1]$ random variable.
Also, given a set~$G\subset\R^2$ that does not 
``surround'' the origin,
we prove that a.s.\ there is an infinite number of $k$'s
such that $kG\cap \Z^2$ is unvisited.
These results suggest that the range of the 
conditioned walk has ``fractal'' behavior.
 \\[.3cm]\textbf{Keywords:} random interlacements, range, transience,
 simple random walk, Doob's $h$-transform
\\[.3cm]\textbf{AMS 2010 subject classifications:}
Primary 60J10. Secondary 60G50, 82C41.
 
\end{abstract}

\section{Introduction and results}
\label{s_introres}

We start by introducing some basic notation and defining
the ``conditioned'' random walk~$\s$, the main object
of study in this paper. Besides being interesting
on its own, this random walk is the main ingredient
in the construction of the two-dimensional random
interlacements of~\cite{CP16,CPV16}
(see also \cite{CT12,DRS14,SLT,Szn10} for the 
higher-dimensional case).

Write~$x\sim y$ if $x$ and $y$ are neighbours in~$\Z^2$.
Let~$(S_n, n\geq 0)$ be two-dimensional simple 
random walk, i.e., the discrete-time Markov chain
with state space~$\Z^2$ and transition
probabilities defined in the following way:
\begin{equation}
\label{def_S}
 P_{xy} = \begin{cases}
           \displaystyle\frac{1}{4}, & \text{ if }x\sim y,\\[0.7em]
             0,
              & \text{ otherwise}.
          \end{cases}
\end{equation}
We assume that all random variables in this paper
are constructed on a common probability space with probability
measure~$\IP$ and we denote by ~$\IE$ the corresponding expectation. 
When no confusion can arise, we will write~$\IP_x$
and~$\IE_x$ for the law and expectation
 of the\footnote{the simple one, or the conditioned
one defined below} 
random walk started from~$x$.
Let
\begin{align}
\tau_0(A) &= \inf\{k\geq 0: S_k\in A\} \label{entrance_t},\\
\tau_1(A) &= \inf\{k\geq 1: S_k\in A\} \label{hitting_t}
\end{align}
be the entrance and the hitting time of the set~$A$ by 
simple random walk~$S$ (we use the convention $\inf \emptyset = +\8$). 
For a singleton $A=\{x\}$, we will write
 $\tau_i(A)=\tau_i(x)$, $i=0,1$,
for short. 
One of the key objects needed to understand the two-dimensional
simple random walk is the potential kernel~$a$, defined by 
\begin{equation}
\label{def_a(x)}
a(x) = \sum_{k=0}^\infty\big(\IP_0[S_k\!=\!0]-\IP_x[S_k\!=\!0]\big).
\end{equation}
It can be shown that the above series indeed converges 
and we have~$a(0)=0$, $a(x)>0$ for $x\neq 0$.
It it straightforward to check that 
 the function~$a$ is harmonic outside the origin, i.e.,
\begin{equation}
\label{a_harm}
 \frac{1}{4}\sum_{y: y\sim x}a(y) = a(x) \quad \text{ for all }
 x\neq 0.
\end{equation}
Also, using~\eqref{def_a(x)} and the Markov property,
 one can easily obtain that
$\frac{1}{4}\sum_{x\sim 0}a(x)=1$, which implies by symmetry
that
\begin{equation}
\label{a(x)=1} 
 a(x)=1 \text{ for all }x\sim 0.
\end{equation}

Observe that~\eqref{a_harm} immediately implies 
that $a(S_{k\wedge \tau_0(0)})$ is a martingale, we will
repeatedly use this fact in the sequel.
Further, one can show that
(with $\gamma=0.5772156\dots$ the Euler-Mascheroni constant)
\begin{equation}
\label{formula_for_a}
 a(x) = \frac{2}{\pi}\ln \|x\| 
  + \frac{2\gamma+ 3\ln 2}{\pi} + O(\|x\|^{-2}) 
\end{equation}
as $x\to\infty$, cf.\ Theorem~4.4.4 of~\cite{LL10}.

Let us define another random walk $(\s_n, n\geq 0)$
on~$\Z^2\setminus \{0\}$ in the following way:
its transition probability matrix  
equals (compare to~\eqref{def_S})
\begin{equation}
\label{def_hatS}
 \hP_{xy} = \begin{cases}
           \displaystyle\frac{a(y)}{4a(x)}, & \text{ if }
            x\sim y, x\neq 0,\\
             0,\phantom{\int\limits^A} & \text{ otherwise.}
          \end{cases}
\end{equation}
It is immediate to see from~\eqref{a_harm} 
that the random walk~$\s$
is indeed well defined.

The walk~$\s$ is the Doob $h$-transform
of the simple random walk, under the condition 
of not hitting the origin
(see Lemma~3.3 of~\cite{CPV16} and its proof).
Let $\htau_0, \htau_1$ be defined as in 
\eqref{entrance_t}--\eqref{hitting_t}, but with~$\s$ in the place of~$S$.
We summarize the basic properties of the walk~$\s$
in the following
\begin{prop}
\label{p_basic_prop}
The following statements hold:
\begin{itemize}
 \item[(i)] The walk~$\s$ is reversible, with the reversible
 measure~$\mu_x:=a^2(x)$.
 \item[(ii)] In fact, it can be represented as a random walk
 on the two-dimensional lattice with conductances
 $\big(a(x)a(y), x,y\in \Z^2, x\sim y\big)$.
 \item[(iii)] Let~$\mathcal{N}$ be the set of the four
neighbours of the origin. Then the process
 $1/a(\s_{n\wedge \htau_0(\mathcal{N})})$ is a 
martingale.
 \item[(iv)]
 The walk $\s$ is transient.
 \item[(v)] 
Moreover, for all $x\neq 0$ 
\begin{equation} 
 \label{escape_from_site}
\IP_x\big[\htau_1(x)<\infty\big] = 1-\frac{1}{2a(x)},
\end{equation}
and for all $x\neq y$, $x,y\neq 0$ 
\begin{equation}
\label{not_hit_site}
 \IP_x\big[\htau_0(y)<\infty\big] = 
 \IP_x\big[\htau_1(y)<\infty\big] = \frac{a(x)+a(y)-a(x-y)}{2a(x)}.
\end{equation}
\end{itemize}
\end{prop}

The statements of Proposition~\ref{p_basic_prop} are not 
novel (they appear already in~\cite{CPV16}), but 
we found it useful to collect them here for the sake
of completeness and also for future reference.
 We will prove Proposition~\ref{p_basic_prop} in the next section.
It is curious to observe that~\eqref{not_hit_site} implies
that, for any~$x$,
 $\IP_x[\htau_1(y)<\infty]$ converges to $\frac{1}{2}$
as $y\to \infty$. As noted in~\cite{CPV16}, this is related
to the remarkable fact that if one conditions on a very
distant site being vacant, then this reduces the intensity
``near the origin'' of the two-dimensional
 random interlacement process by a factor of four. 

Let~$\|\cdot\|$ be the Euclidean norm. Define the (discrete)
ball
\[
 \B(x,r) = \{y\in \Z^2: \|y-x\|\leq r\}
\]
(note that this definition works for \emph{all} 
$x\in\R^2$ and $r\in\R_+$),
and
abbreviate $\B(r):=\B(0,r)$.
The (internal) boundary of $A\subset\Z^2$ is defined by
\[
 \partial A = \{x\in A: \text{there exists }y\in \Z^2\setminus A
 \text{ such that }x\sim y\}.
\]

Now we introduce some more notation and state the main results.
For a set $T\subset\Z_+$ (thought of as a set of time moments) let 
\[
\s_T=\bigcup_{m\in T}\big\{\s_m\big\}
\]
be the \emph{range} of the walk~$\s$ with respect to that set.
For simplicity,
we assume in the following that the walk~$\s$ starts 
at a fixed neighbour~$x_0$ of the origin, 
and we write $\IP$ for $\IP_{x_0}$
(it is, however, clear that our results hold for any
fixed starting position of the walk).
For a nonempty and finite set $A\subset \Z^2$, let us consider random variables
\begin{align*} 
 \RR(A) &= \frac{\big|A\cap \s_{[0,\infty)}\big|}{|A|} ,\\
 \V(A) &= \frac{\big|A\setminus \s_{[0,\infty)}\big|}{|A|}
 = 1-\RR(A);
\end{align*}
that is, $\RR(A)$ (respectively, $\V(A)$) is the proportion 
of visited (respectively, unvisited) sites of~$A$ by the walk~$\s$.
Let us also abbreviate, for $M_0 > 0$,
\begin{equation}
\label{df_ell_A}
 \ell^{(n)}_A=  |A|^{-1}
 \max_{y\in A} 
\big|A\cap \B\big(y,{\textstyle\frac{n}{\ln^{M_0}n}}\big)\big|.
\end{equation}
Our main result is the following
\begin{theo}
\label{t_main_res}
Let $M_0>0$ be a fixed constant, and
assume that $A\subset \B(n)\setminus \B(n\ln^{-M_0}n)$.
Then, for all $s\in [0,1]$, we have,
with positive constants $c_{1,2}$ depending only on~$M_0$,
\begin{equation}
\label{main_res}
 \big|\IP[\V(A)\leq s] - s\big| 
\leq c_1\Big(\frac{\ln\ln n}{\ln n}\Big)^{1/3}
  + c_2\ell^{(n)}_A \Big(\frac{\ln\ln n}{\ln n}\Big)^{-2/3},
\end{equation}
and the same result holds with~$\RR$ on the place of~$\V$.
\end{theo}

The above result means that 
if~$A\subset \B(n)\setminus \B(n\ln^{-M_0}n)$
is ``big enough and well distributed'', then the proportion 
of visited sites has approximately Uniform$[0,1]$ distribution.
In particular, one can obtain the following
\begin{cor}
\label{cor_open}
Assume that $D\subset \R^2$
is a bounded open set.
Then
both sequences $(\RR(nD\cap\Z^2),n\geq 1)$ 
and $(\V(nD\cap\Z^2),n\geq 1)$
converge in distribution to the Uniform$[0,1]$
random variable.
\end{cor}
Indeed, it is straightforward to obtain it from 
Theorem~\ref{t_main_res}
 since $|nD\cap\Z^2|$ is of order $n^2$
as $n\to\infty$ (note that~$D$ contains a disk), and so
 $\ell^{(n)}_{nD\cap\Z^2}$ will be of 
order~$\ln^{-2M_0}n$. 
Observe that we can cut out $\B(n\ln^{-M_0}n)$ from~$nD$ 
without doing any harm to the
limit theorem, since formally we need
$A\subset \B(n)\setminus \B(n\ln^{-M_0}n)$
 in order to apply Theorem~\ref{t_main_res}.
Then, we can choose~$M_0$ large enough such that 
the right-hand side of~\eqref{main_res} goes to~$0$.

Also, we prove that the range of~$\s$ contains many
``big holes''. 
To formulate this result, we need the following
\begin{df}
\label{df_surround}
 We say that a set $G\subset \R^2$ \emph{does not surround
 the origin}, if
\begin{itemize}
 \item there exists $c_1>0$ such that $G\subset\B(c_1)$,
 i.e., $G$ is bounded;
 \item there exist $c_2>0$, $c_3>0$, and a function 
 $f=(f_1,f_2): [0,1]\mapsto \R^2$
such that $f(0)=0$, $\|f(1)\|= c_1$, 
$|f'_1(s)|+|f'_2(s)|\leq c_2$
for all $s\in [0,1]$, and
\[
 \inf_{s\in[0,1], y\in G} \|(f_1(s),f_2(s))-y\|\geq c_3,
\]
i.e., one can escape from the origin to infinity 
along a path which is uniformly away from~$G$.
\end{itemize}
\end{df}

Then, we have
\begin{theo}
\label{t_bigholes}
 Let $G\subset \R^2$ be a set that does not
 surround the origin. Then,
\begin{equation}
\label{eq_bigholes}
 \IP\big[nG\cap \s_{[0,\infty)} = \emptyset 
\text{ for infinitely many }n\big] = 1. 
\end{equation}
\end{theo}

Theorem~\ref{t_bigholes} invites the following
\begin{rem}
\label{rem_filled}
A natural question to ask is whether there are also ``big''
\emph{completely filled} subsets 
of~$\Z^2$, that is, 
if a.s.\ there are infinitely many~$n$ such that
$(nG\cap \Z^2)\subset \s_{[0,\infty)}$, for $G\subset\R^2$ 
being, say, a disk. It is not difficult to see that the answer 
to this question is ``no''. We do not give all details,
but the reason for this is that, informally, one $\s$-trajectory
corresponds to the two-dimensional random interlacements
of~\cite{CPV16} ``just above'' the level $\alpha=0$.
Then, as in Theorem~2.5~(iii) (inequality~(22)) of~\cite{CPV16},
it is possible to show that, with \emph{any} fixed $\delta>0$,
\[
 \IP\big[(nG\cap \Z^2)\subset \s_{[0,\infty)}\big]
     \leq n^{-2+\delta}
\]
for all large enough~$n$; our claim then follows from the 
(first) Borel-Cantelli lemma.
\end{rem}

We also establish some additional properties
of the conditioned walk~$\s$, which will be important
for the proof of Theorem~\ref{t_bigholes} and 
are of independent interest. Consider an irreducible Markov chain.
Recall that a set is called
\emph{recurrent} with respect to the Markov chain, if it is
visited infinitely many times almost surely; a set is called \emph{transient},
if it is visited only finitely many times almost surely.
 It is clear that
any nonempty set is recurrent
with respect to a recurrent Markov chain, and every finite
set is transient with respect to a transient Markov chain.
Note that, in general, a set can be neither recurrent nor 
transient --- think e.g.\ of the simple random walk on a binary
tree, fix a neighbour of the root and consider the set of vertices of the tree connected to the root
through this fixed neighbour.

In many situations it is possible to characterize completely
the recurrent and transient sets, as well as to answer the question
if any set must be either recurrent or transient. For example,
for the simple random walk in~$\Z^d$, $d\geq 3$, each set is either recurrent or transient and
the characterization
is provided by the \emph{Wiener's test} 
(see e.g.\ Corollary~6.5.9 of~\cite{LL10}), 
formulated in terms of capacities of
intersections of the set with exponentially growing annuli.
Now, for the conditioned two-dimensional walk~$\s$ the characterization of recurrent and transient sets
is particularly simple:
\begin{theo}
\label{t_rec_trans}
A set 
$A\subset\Z^2$ is recurrent with respect to~$\s$
if and only if~$A$ is infinite. 
\end{theo}

Next, we recall that a Markov chain
has the \emph{Liouville property}, 
see e.g.\ Chapter~IV of~\cite{Woess09}, 
if all bounded harmonic
(with respect to that Markov chain) functions are constants.
Since Theorem~\ref{t_rec_trans} implies that every
set must be recurrent or transient, we obtain the following
result as its corollary:
\begin{theo}
\label{t_Liouville}
 The conditioned two-dimensional walk~$\s$ has the Liouville property.
\end{theo}
These two results, besides being of interest on their own,
will also be operational in the proof of Theorem~\ref{t_bigholes}.

\section{Some auxiliary facts 
and proof of Proposition~\ref{p_basic_prop}}
\label{s_aux}

For $A\subset \Z^d$, recall that $\partial A$ denotes its internal boundary.
We abbreviate $\tau_1(R) = \tau_1(\partial \B(R)).$
We will consider, with a slight abuse of notation,
the function 
\[
a(r)=\frac{2}{\pi}\ln r + \frac{2\gamma+ 3\ln 2}{\pi} 
\]
of a \emph{real} argument~$r\geq 1$. 
To explain why this notation is convenient,
observe that, 
due to~\eqref{formula_for_a}, we may write,
for the case when (say) $2\|x\|\leq r$ and
 as $r\to\infty$,
\begin{equation}
\label{real_a}
 \sum_{y\in\partial \B(x,r)} \nu(y)a(y) = a(r) 
+ O\Big(\frac{\|x\|\vee 1}{r}\Big)
\end{equation}
for \emph{any} probability measure~$\nu$ on $\partial \B(x,r)$.

For all $x \in \Z^2$ and $R\geq 1$ 
such that
$x,y \in  \B(R/2)$ and $x\neq y$,  we have
\begin{equation}
 \label{nothit_r_dim2}
\IP_x[\tau_1(R)< \tau_1(y)] = 
\frac{a(x-y)}{a(R)+O\big(R^{-1}(\|y\|\vee 1)\big)},
\end{equation}
as $R\to \infty$. This is an easy consequence of the 
optional stopping theorem applied to 
the martingale~$a(S_{n\wedge\tau_0(y)}-y)$, together
 with~\eqref{real_a}.
Also, an application of
the optional stopping theorem
to the martingale $1/a(\s_{n\wedge \htau_0(\mathcal{N})})$
yields
\begin{equation}
\label{hitting_condS}
 \IP_x[\htau_1(R)<\htau_1(r)] 
 = \frac{(a(r))^{-1}-(a(x))^{-1}+O(R^{-1})}
 {(a(r))^{-1}-(a(R))^{-1}+O(r^{-1})},
\end{equation}
for $1<r<\|x\|<R<\infty$.
Sending $R$ to infinity in~\eqref{hitting_condS}
we see that for $1\leq r\leq \|x\|$
\begin{equation}
\label{escape_condS}
 \IP_x[\htau_1(r)=\infty] = 1-\frac{a(r)+O(r^{-1})}{a(x)}.
\end{equation}

We need the fact 
that~$S$ conditioned on hitting $\partial\B(R)$ before~$0$
is almost indistinguishable from~$\s$ .
For $A \subset \Z^2$, let~$\Gamma^{(x)}_A$ denote the set of all 
 finite nearest-neighbour
trajectories that start at~$x\in A\setminus\{0\}$ 
and end when entering~$\partial A$ for the first time.
 For~$V\subset \Gamma^{(x)}_A$ write
 $S\in V$ if there exists~$k$ such that 
$(S_0,\ldots,S_k)\in V$ (and the same for
the conditioned walk~$\s$). We write $\Gamma^{(x)}_{0,R}$ for $\Gamma^{(x)}_{B(R)}$.
\begin{lem}
\label{l_relation_S_hatS}
Assume that $V\subset \Gamma^{(x)}_{0,R}$; then we have 
\begin{equation}
\label{eq_relation_S_hatS}
\IP_x[S\in V\mid \tau_1(R)< \tau_1(0)]
 =\IP_x[\s \in V] \big(1+O((R \ln R)^{-1})\big).  
\end{equation}
\end{lem}

\begin{proof}
 This is Lemma~3.3 (i) of~\cite{CPV16}.
\end{proof}

If $A\subset A'$ are (finite) subsets
of~$\Z^2$, then the \emph{excursions}
between~$\partial A$ and~$\partial A'$ are
pieces of nearest-neighbour trajectories
that begin on~$\partial A$ and end on~$\partial A'$,
see Figure~\ref{f_excursions}, 
which is, hopefully, self-explanatory.  We refer  
to Section~3.4 of~\cite{CPV16} for formal definitions.

\begin{figure}
\begin{center}
\includegraphics{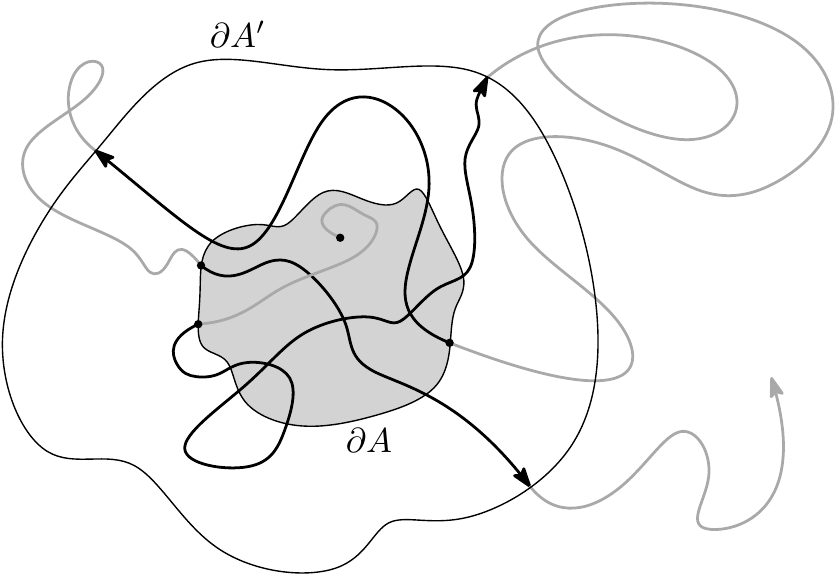}
\caption{Excursions (pictured as bold pieces of trajectories)
of random walks between~$\partial A$ and $\partial A'$.}
\label{f_excursions}
\end{center}
\end{figure}

\begin{proof}[Proof of Proposition~\ref{p_basic_prop}.]
It is straightforward to check (i)--(iii) directly,
we leave this task for the reader. 
Item~(iv) (the transience) 
follows from~(iii) and Theorem~2.5.8 of~\cite{MPW17}.

As for~(v), we first observe that~\eqref{escape_from_site}
is a consequence of~\eqref{not_hit_site}, although it is of course
also possible to prove it directly, see Proposition~2.2
of~\cite{CPV16}.
Indeed, using~\eqref{def_hatS} and then \eqref{not_hit_site}, \eqref{a_harm} and \eqref{a(x)=1}, one can write 
\begin{align*}
 \IP_x\big[\htau_1(x)<\infty\big] 
&= \frac{1}{4a(x)}\sum_{y\sim x} a(y)
  \IP_y\big[\htau_1(x)<\infty\big]\\
&= \frac{1}{4a(x)}\sum_{y\sim x} \frac{1}{2}\left(a(y) + a(x) -a(y-x)\right)   \\
&= 1 - \frac{1}{2a(x)}.
\end{align*}

Now, to prove~\eqref{not_hit_site}, we essentially
use the approach of Lemma~3.7 of~\cite{CPV16}, although
here the calculations are simpler.
Let us define (note that all the probabilities below
are for the simple random walk~$S$)
\begin{align*}
 h_1&=\IP_x[\tau_1(0) <\tau_1(R)],\\
h_2&= \IP_x[\tau_1(y) <\tau_1(R)],\\
q_{12}&=\IP_0[\tau_1(y) <\tau_1(R)],\\
q_{21}&=\IP_y[\tau_1(0)<\tau_1(R)],\\
p_1&=\IP_x[\tau_1(0) <\tau_1(R) \wedge \tau_1(y)],\\
p_2&=\IP_x[\tau_1(y) <\tau_1(R) \wedge \tau_1(0)],
\end{align*}
see Figure~\ref{f_pqh12}.

\begin{figure}
\begin{center}
\includegraphics{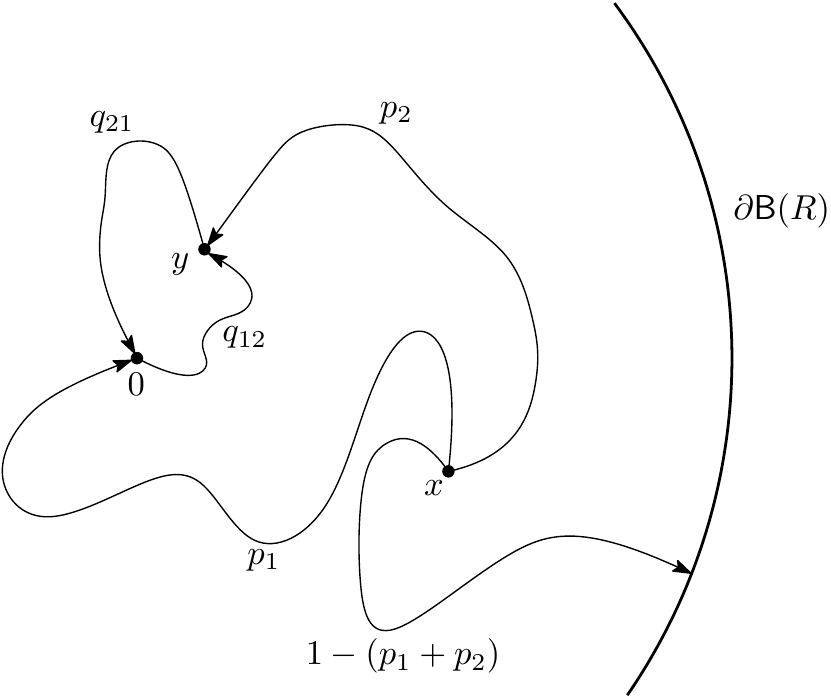}
\caption{Trajectories for the probabilities of interest.}
\label{f_pqh12}
\end{center}
\end{figure}
Using~\eqref{nothit_r_dim2} (and in addition the Markov property and \eqref{a_harm} for \eqref{q12})  we have for $x,y \neq 0$, $x \neq y$
\begin{align}
 h_1&=1-\frac{a(x)}{a(R)+O(R^{-1})},
  \label{h1}\\
 h_2&=1-\frac{a(x-y)}{a(R)+O(R^{-1}\|y\|)},
\label{h2}\\
 q_{12}&=1-\frac{a(y)}{a(R)+O(R^{-1}\|y\|)},
 \label{q12}\\
 q_{21}&=1-\frac{a(y)}{a(R)+O(R^{-1})},
 \label{q21}
\end{align}
which implies that 
\begin{align}
\lim_{R\to\infty} (1-h_1)a(R)&=a(x),
  \label{lim_h1}\\
 \lim_{R\to\infty} (1-h_2)a(R)&=a(x-y),
\label{lim_h2}\\
\lim_{R\to\infty} (1-q_{12})a(R) &= a(y),
 \label{lim_q12}\\
 \lim_{R\to\infty} (1-q_{21})a(R) &=a(y).
 \label{lim_q21}
\end{align}

Observe that, due to the Markov property, it holds that
\begin{align*}
 h_1 &= p_1 + p_2q_{21},\\
 h_2 &= p_2 + p_1q_{12}.
\end{align*}
Solving these equations with respect to $p_1,p_2$, we
obtain
\begin{align}
 p_1 &= \frac{h_1-h_2q_{21}}{1-q_{12}q_{21}},
\label{expr_p1}\\
 p_2 &= \frac{h_2-h_1q_{12}}{1-q_{12}q_{21}}.
\label{expr_p2}
\end{align}
Let us denote 
\begin{equation}
\label{bar_all}
\hh_1=1-h_1, \quad \hh_2=1-h_2,\quad
\hq_{12}=1-q_{12},\quad \hq_{21}=1-q_{21}. 
\end{equation}
Next, using
Lemma~\ref{l_relation_S_hatS}, we have that 
\begin{align}
 \IP_x[\htau_1(y)<\htau_1(R)] &= \IP_x[\tau_1(y)<\tau_1(R)\mid \tau_1(R)<\tau_1(0)]
\big(1+o(R^{-1})\big)\nonumber\\
 &=\frac{\IP_x[\tau_1(y)<\tau_1(R)<\tau_1(0)]}{\IP_x[\tau_1(R)<\tau_1(0)]}
\big(1+o(R^{-1})\big)\nonumber\\
&=\frac{p_2(1-q_{21})}{1-h_1} \big(1+o(R^{-1})\big)\nonumber\\
&= \frac{(h_2-h_1q_{12})(1-q_{21})}{(1-q_{12}q_{21})(1-h_1)}
\big(1+o(R^{-1})\big)\nonumber\\
&= \frac{(\hh_1+\hq_{12}-\hh_2-\hh_1\hq_{12})\hq_{21}}
{(\hq_{12}+\hq_{21}-\hq_{12}\hq_{21})\hh_1} \big(1+o(R^{-1})\big).
\label{hat_hq}
\end{align}
Since $\IP_x[\htau_1(y)<\infty]=\lim_{R\to\infty}
\IP_x[\htau_1(y)<\htau_1(R)]$,
using~\eqref{lim_h1}--\eqref{lim_q21} we obtain~\eqref{not_hit_site}
(observe that the ``product'' terms in~\eqref{hat_hq}
are of smaller order and will disappear in the limit).
\end{proof}

%

We now use the ideas contained in the last proof to obtain
some refined bounds on the hitting probabilities
for excursions of the conditioned walk.

\begin{figure}
\begin{center}
\includegraphics{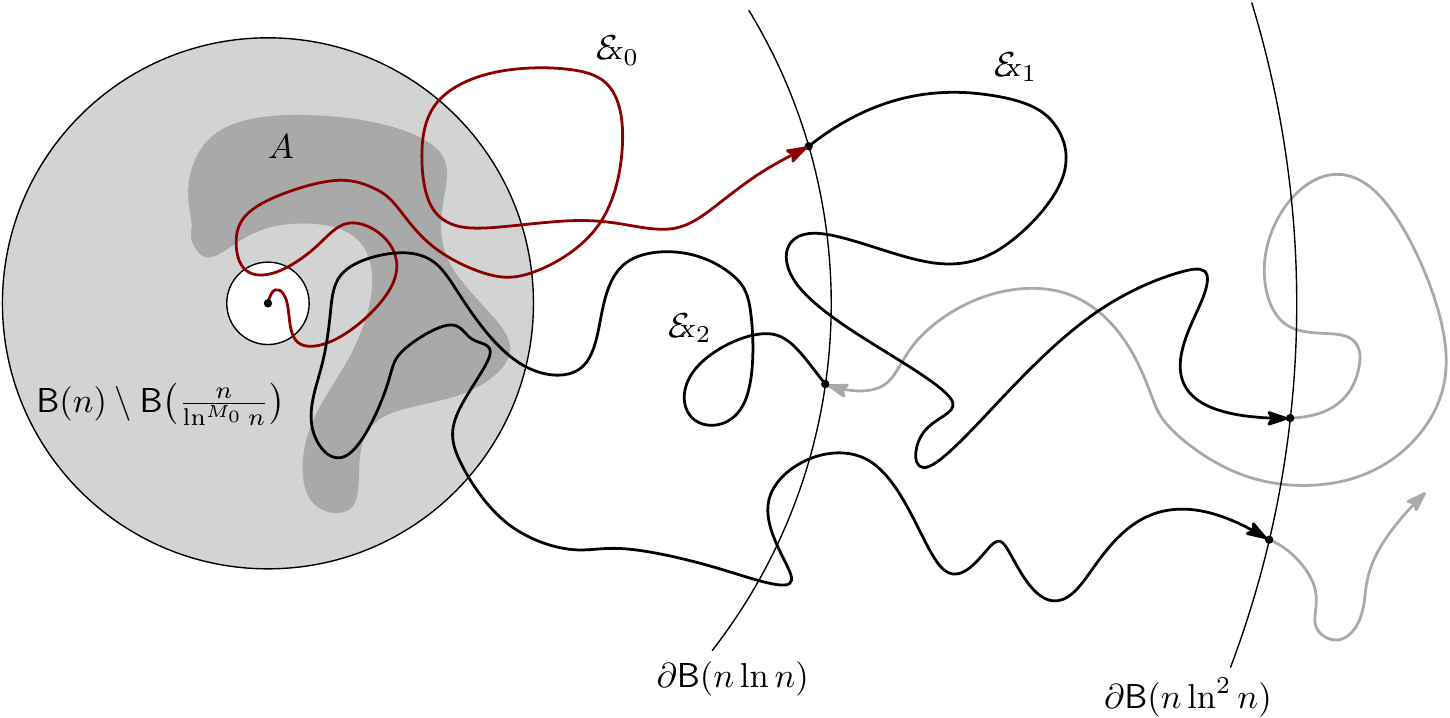}
\caption{Excursions and their visits to $A$}
\label{f_excur_disks}
\end{center}
\end{figure}

Let us assume that $\|x\|\geq n\ln^{-M_0}n$ and $y\in A$,
where the set~$A$ is as in Theorem~\ref{t_main_res}.
Also, abbreviate $R=n\ln^2 n$.
\begin{lem}
\label{l_refined_probs} 
In the above situation, we have
\begin{equation}
\label{gen_exc_hit} 
 \IP_x[\htau_1(y)<\htau_1(R)] = \big(1+O(\ln^{-3} n)\big)
\frac{a(x)a(R)+a(y)a(R)-a(x-y)a(R)-a(x)a(y)}{a(x)(2a(R)-a(y))}.
\end{equation}
\end{lem}

\begin{proof}
%
%
This is essentially the same calculation as in the 
proof of~\eqref{not_hit_site}, with the following
difference: after arriving to the expression~\eqref{hat_hq},
instead of sending~$R$ to infinity
(which conveniently ``kills'' many terms there),
we need to carefully deal with all the $O$'s.
Specifically, we reuse notations~\eqref{h1}--\eqref{q21}
and~\eqref{bar_all}, then write
\begin{align}
 \IP_x[\htau_1(y)<\htau_1(R)] &= \IP_x[\tau_1(y)<\tau_1(R)\mid \tau_1(R)<\tau_1(0)]
\big(1+O(n^{-1})\big)\nonumber\\
&= \frac{B_1}{B_2}\big(1+o(n^{-1})\big),
\label{B1/B2}
\end{align}
where (observe that, since $\|y\|\leq n$ and $R=n\ln^2 n$,
we have $a(R)+O(R^{-1}\|y\|)
=a(R)+O(\ln^{-2} n)=a(R)(1+O(\ln^{-3} n)$)
\begin{align*}
 B_1 &= (\hh_1+\hq_{12}-\hh_2-\hh_1\hq_{12})\hq_{21}\\
&= \frac{a(y)}{a(R)+O(R^{-1})}
\Big(\frac{a(x)}{a(R)+O(R^{-1})}+\frac{a(y)}{a(R)+O(R^{-1}\|y\|)}\\
& \qquad \qquad \qquad
-\frac{a(x-y)}{a(R)+O(R^{-1}\|y\|)}
-\frac{a(x)a(y)}{(a(R)+O(R^{-1}))(a(R)+O(R^{-1}\|y\|))}\Big)\\
&= \big(1+O((R\ln R)^{-1})\big)
\frac{a(y)}{a(R)} \cdot
\frac{a(x)a(R)+a(y)a(R)-a(x-y)a(R)-a(x)a(y)}{(1+O(\ln^{-3} n))a^2(R)}\\
&= \big(1+O(\ln^{-3} n)\big)
\frac{a(y)}{a(R)} \cdot
\frac{a(x)a(R)+a(y)a(R)-a(x-y)a(R)-a(x)a(y)}{a^2(R)},
\end{align*}
and
\begin{align*}
 B_2 &= (\hq_{12}+\hq_{21}-\hq_{12}\hq_{21})\hh_1\\
&= \frac{a(x)}{a(R)+O(R^{-1})}
\Big(\frac{a(y)}{a(R)+O(R^{-1}\|y\|)}+\frac{a(y)}{a(R)+O(R^{-1})}\\
& \qquad \qquad \qquad \qquad \qquad
-\frac{a^2(y)}{(a(R)+O(R^{-1}))(a(R)+O(R^{-1}\|y\|))}\Big)\\
&=\big(1+O((R\ln R)^{-1})\big)
\frac{a(x)}{a(R)} \cdot 
\frac{2a(y)a(R)-a^2(y)+O(\ln^{-1}n)}{(1+O(\ln^{-3} n))a^2(R)}\\
&= \big(1+O(\ln^{-3} n)\big)
\frac{a(x)}{a(R)} \cdot \frac{2a(y)a(R)-a^2(y)}{a^2(R)}.
\end{align*}
We insert the above back to~\eqref{B1/B2}
and note that the factor $\frac{a(y)}{a^3(R)}$ cancels
to obtain~\eqref{gen_exc_hit}.
\end{proof}

\section{Proofs of the main results}
\label{s_proofs}
We start with
\begin{proof}[Proof of Theorem~\ref{t_main_res}]
First, we describe informally the idea of the proof.
We consider the visits to the set $A$ during excursions 
of the walk from $\partial\B(n\ln n)$ to
$\partial\B(n\ln^2 n)$, see Figure~\ref{f_excur_disks}.
The crucial argument is the following: the randomness of~$\V(A)$
comes from the \emph{number} of excursions 
  and not from the excursions themselves.
If the number of excursions is around $c\times\frac{\ln n}{\ln\ln n}$,
then it is possible to show (using a standard
weak-LLN argument) that the proportion of uncovered sites
in~$A$ is \emph{concentrated} around~$e^{-c}$.
On the other hand, that number of excursions can be modeled
roughly as $Y \times\frac{\ln n}{\ln\ln n}$, where~$Y$ is an
Exponential($1$) random variable. Then,
$\IP[\V(A)\leq s]\approx \IP[Y\geq \ln s^{-1}]=s$,
as required.

We now give a rigorous argument.
Let $\hH$ be the conditional entrance measure
for the (conditioned) walk~$\s$, i.e.,
\begin{equation}
\label{df_hatH}
 \hH_A(x,y) = \IP_x\big[\s_{\htau_1(A)}=y\mid \htau_1(A)<\infty\big].
\end{equation}
Let us first denote the initial piece of the trajectory
by $\Ex_0=\s_{[0,\htau(n\ln n)]}$.
Then, we consider a \emph{Markov chain} $(\Ex_k,k\geq 1)$
of excursions between $\partial\B(n\ln n)$
and $\partial\B(n\ln^2 n)$, defined
 in the following
way: for $k\geq 2$ the initial site of~$\Ex_k$ is chosen according
to the measure~$\hH_{\B(n\ln n)}(z_{k-1},\cdot)$,
where $z_{k-1}\in\partial\B(n\ln^2 n)$ is the last
site of the excursion~$\Ex_{k-1}$;
also, the initial site of~$\Ex_1$ is 
the last site of $\Ex_0$; the weights of trajectories 
are chosen according to~\eqref{def_hatS}
(i.e., each excursion is an $\s$-walk trajectory).
It is important to observe that one may couple
 $(\Ex_k,k\geq 1)$ with the ``true'' excursions of the walk~$\s$
in an obvious way: one just picks the excursions subsequently,
each time tossing a coin to decide if the walk returns
to $\B(n\ln n)$.

Let 
\[
  \psi_n = \min_{x\in\partial\B(n\ln^2 n)}\IP_x[\htau(n\ln n)=\infty]
\]
be the \emph{minimal}  probability to avoid $\B(n\ln n)$,
starting at sites of $\partial\B(n\ln^2 n)$.
Using~\eqref{escape_condS} it is straightforward to obtain that
\[
 \IP_x[\htau_1(n\ln n)=\infty] = \frac{\ln\ln n}{\ln n + 2\ln\ln n}\big(1+O(n^{-1})\big)
\]
for any $x\in \partial\B(n\ln^2 n)$, and so it also holds that
\begin{equation}
\label{oc_psi_n}
 \psi_n = \frac{\ln\ln n}{\ln n + 2\ln\ln n}\big(1+O(n^{-1})\big). 
\end{equation}
Let us consider a sequence of i.i.d.\ random 
variables $(\eta_k,k\geq 0)$ such that
$\IP[\eta_k=1]=1-\IP[\eta_k=0]=\psi_n$.
Let $\hN  = \min\{k:\eta_k = 1\}$, so that~$\hN$ is a Geometric
random variable with mean~$\psi_n^{-1}$.
Now, \eqref{oc_psi_n} implies that 
$\IP_x[\htau(n\ln n)=\infty]-\psi_n\leq O\big(\frac{\ln\ln n}{n\ln n}\big)$
for any $x\in \partial\B(n\ln^2 n)$,
so it is clear\footnote{Let $(Z_n,n\geq 1)$ 
be a sequence of $\{0,1\}$-valued
 random variables
adapted to a filtration $(\FFF_n,n\geq 1)$ and such that
$\IP[Z_{n+1}=1\mid \FFF_n]\in [p,p+\eps]$ a.s..
Then it is elementary to obtain that
the total variation distance between the random 
variable $\min\{k: Z_k=1\}$ and the Geometric random variable 
with mean~$p^{-1}$ is bounded above by $O(\eps/ p)$.}
 that~$\hN $ can be coupled with the 
actual number of excursions~$N$ in such a way that 
$N\leq \hN $ a.s.\ and 
\begin{equation}
\label{N_neq_hatN}
 \IP[N \neq \hN ] \leq 
 O(n^{-1}).
\end{equation}
Note that this construction preserves the independence
of~$\hN$ from the excursion sequence $(\Ex_k,k\geq 1)$
itself.

Define 
\begin{align*}
 \RR^{(k)} &= \frac{\big|A\cap 
(\Ex_0\cup \Ex_1\cup \ldots \cup \Ex_k)\big|}{|A|},
\intertext{and}
 \V^{(k)} &= \frac{\big|A\setminus 
(\Ex_0\cup \Ex_1\cup \ldots \cup \Ex_k)\big|}{|A|}=1-\RR^{(k)}
\end{align*}
to be the proportions of visited and unvisited sites in~$A$
with respect to the first~$k$ excursions
together with the initial piece~$\Ex_0$.

Now, it is straightforward to check
that~\eqref{gen_exc_hit} implies that, 
for any $x\in \partial\B(n\ln n)$ and $y \in A$
\begin{equation}
\label{hit_y}
 \IP_x\big[\htau_1(y)<\htau_1(n\ln^2 n)\big] = \frac{\ln\ln n}{\ln n}
   \Big(1+O\Big(\frac{\ln\ln n}{\ln n}\Big)\Big),
\end{equation}
and, for $y,z\in \B(n)\setminus\B\big(\frac{n}{2\ln^{M_0}n}\big)$ 
such that $\|y-z\| = n/b$ with 
$b\leq 2\ln^{M_0}n$
\begin{equation}
\label{hit_z}
  \IP_z\big[\htau_1(y)<\htau_1(n\ln^2 n)\big] = 
\frac{2\ln\ln n + \ln b}{\ln n}
   \Big(1+O\Big(\frac{\ln\ln n}{\ln n}\Big)\Big).
\end{equation}
Indeed, first, observe that the factor~$B_2$
in~\eqref{gen_exc_hit}
is, in both cases,
\begin{equation}
\label{ocenka_B2}
 a(x)(2a(R)-a(y)) = \Big(\frac{2}{\pi}\Big)^2
  \ln^2 n + O\big(\ln n\ln\ln n\big).
\end{equation}
As for the factor~$B_1$, we have
\begin{align*}
B_1 &=  a(x)a(R)+a(y)a(R)-a(x-y)a(R)-a(x)a(y)\\
 &= (a(x)-a(x-y))a(R) - (a(R)-a(x))a(y)\\
 &= O\big((\ln n)^{-1}\big)\times O(\ln n)
 + \Big(\frac{2}{\pi}\ln\ln n + o(n^{-2})\Big)
 \times \Big(\frac{2}{\pi}\ln n + O(\ln\ln n)\Big)\\
&= \Big(\frac{2}{\pi}\Big)^2
  \ln n\ln\ln n + O\big((\ln\ln n)^2\big)
\end{align*}
in the case of~\eqref{hit_y}, and
(writing also $\|z\|=n/c$ with 
$(2\ln^{M_0} n)^{-1}\leq c\leq 1$)
\begin{align*}
B_1 &=  a(z)a(R)+a(y)a(R)-a(z-y)a(R)-a(z)a(y)\\
 &= (a(z)-a(z-y))a(R) - (a(R)-a(z))a(y)\\
 &= \frac{2}{\pi}\Big(-\ln c +\ln b + o(n^{-1})\Big)
 \times \frac{2}{\pi}\Big(\ln n + O(\ln\ln n)\Big)\\
& \qquad + \frac{2}{\pi}\Big(2\ln\ln n +\ln c + o(n^{-1})\Big)
 \times \Big(\frac{2}{\pi}\ln n + O(\ln\ln n)\Big)\\
&= \Big(\frac{2}{\pi}\Big)^2
  \ln n\times(2\ln\ln n + \ln b) + O\big((\ln\ln n)^2\big)
\end{align*}
in the case of~\eqref{hit_z};
with~\eqref{ocenka_B2} we then 
 obtain~\eqref{hit_y}--\eqref{hit_z}.

For $y\in A$ and a fixed~$k\geq 1$ consider the random variable
\[
 \xi_y^{(k)} = \1{y\notin\Ex_0\cup \Ex_1\cup \ldots \cup \Ex_k},
\]
so that $\V^{(k)} = |A|^{-1}\sum_{y\in A}\xi_y^{(k)}$.
Now \eqref{hit_y}
implies that, for all $j\geq 1$,
\[
\IP[y\notin \Ex_j] = 1-\frac{\ln\ln n}{\ln n}
   \Big(1+O\Big(\frac{\ln\ln n}{\ln n}\Big)\Big),
\]
and~\eqref{hit_z} implies that 
\[
 \IP[y\notin \Ex_0\cup\Ex_1] = 1 - O\Big(\frac{\ln\ln n}{\ln n}\Big)
\]
for any~$y\in A$.
Let $\mu_y^{(k)}=\IE \xi_y^{(k)}$. Then we have
\begin{align}
\mu_y^{(k)} &= \IP[y\notin\Ex_0\cup \Ex_1\cup \ldots \cup \Ex_k]
\nonumber\\
  &= \Big(1-O\Big(\frac{\ln\ln n}{\ln n}\Big)\Big)
\times 
\Bigg(\Big(1-\frac{\ln\ln n}{\ln n}
   \Big(1+O\Big(\frac{\ln\ln n}{\ln n}\Big)\Big)\Big)
\Bigg)^{k-1}\nonumber\\
 &= \exp\Big(
  -k\frac{\ln\ln n}{\ln n}
\Big(1+O\Big(k^{-1}+\frac{\ln\ln n}{\ln n}\Big)\Big)\Big).
\label{est_mu_k}
\end{align}
Next, we need to estimate the covariance of~$\xi_y^{(k)}$
and~$\xi_z^{(k)}$ in case $\|y-z\|\geq n\ln^{-M_0}n$.
First note that, for any $x\in \partial\B(n\ln n)$
\begin{align*}
 \IP_x\big[\{y,z\}\cap \Ex_1 = \emptyset \big] &=
 1-\IP_x[y\in\Ex_1] - \IP_x[z\in\Ex_1]
+ \IP_x\big[\{y,z\}\subset \Ex_1 \big]\\
&= 1- 2 \frac{\ln\ln n}{\ln n}
\Big(1+O\Big(\frac{\ln\ln n}{\ln n}\Big)\Big)
+ \IP_x\big[\{y,z\}\subset \Ex_1 \big]
\end{align*}
by~\eqref{hit_y};
also, since
\begin{align*}
 \big\{\htau_1(y)<\htau_1(z)<\htau_1(n\ln^2 n)\big\}
&\subset
  \big\{\htau_1(y)<\htau_1(n\ln^2 n),\\
 &\qquad\qquad\s_k=z \text{ for some } 
\htau_1(y)<k<\htau_1(n\ln^2 n)\big\}
\end{align*}
 from~\eqref{hit_y}--\eqref{hit_z} we obtain
\begin{align*}
\IP_x\big[\{y,z\}\subset \Ex_1 \big] &=
 \IP_x\big[\max\{\htau_1(y),\htau_1(z)\}<\htau_1(n\ln^2 n)\big]
\\
&= \IP_x\big[\htau_1(y)<\htau_1(z)<\htau_1(n\ln^2 n)\big]
+  \IP_x\big[\htau_1(z)<\htau_1(y)<\htau_1(n\ln^2 n)\big]\\
&\leq \IP_x\big[\htau_1(y)<\htau_1(n\ln^2 n)\big]
  \IP_y\big[\htau_1(z)<\htau_1(n\ln^2 n)\big]\\
& \qquad + \IP_x\big[\htau_1(z)<\htau_1(n\ln^2 n)\big]
  \IP_z\big[\htau_1(y)<\htau_1(n\ln^2 n)\big]\\
&\leq 2\frac{\ln\ln n}{\ln n} \times \frac{(2+M_0)\ln\ln n}{\ln n}
\Big(1+O\Big(\frac{\ln\ln n}{\ln n}\Big)\Big)\\
 &= O\Big(\Big(\frac{\ln\ln n}{\ln n}\Big)^2\Big).
\end{align*}
Therefore, similarly to~\eqref{est_mu_k} we obtain  
 \begin{align*}
  \IE(\xi_y^{(k)}\xi_z^{(k)})
&= \exp\Big(-2k\frac{\ln\ln n}{\ln n}
\Big(1+O\Big(k^{-1}+\frac{\ln\ln n}{\ln n}\Big)\Big)\Big),
 \end{align*}
which, together with~\eqref{est_mu_k},
  implies after some elementary calculations
that, for all~$y,z\in A$ such that $\|y-z\|\geq n\ln^{-M_0}n$
\begin{equation}
\label{est_cov}
 \cov (\xi_y^{(k)},\xi_z^{(k)}) 
= O\Big(\frac{\ln\ln n}{\ln n}\Big)
\end{equation}
uniformly in~$k$, 
since
\[
\Bigg(\frac{\ln\ln n}{\ln n} + k\Big(\frac{\ln\ln n}{\ln n}\Big)^2
\Bigg) \exp\Big(-2k\frac{\ln\ln n}{\ln n}\Big)
= O\Big(\frac{\ln\ln n}{\ln n}\Big)
\]
uniformly in~$k$.
Recall the notation~$\ell^{(n)}_A$ from~\eqref{df_ell_A}.
Now, using Chebyshev's inequality, we write
\begin{align}
 \lefteqn{
\IP\Big[\Big||A|^{-1}\sum_{y\in A}(\xi_y^{(k)}
    -\mu_y^{(k)})\Big|>\eps\Big]
}\nonumber\\
&\leq (\eps|A|)^{-2} \Var \Big(\sum_{y\in A}\xi_y^{(k)}\Big)
\nonumber\\
&= (\eps|A|)^{-2} \sum_{y,z\in A}\cov (\xi_y^{(k)},\xi_z^{(k)})
\nonumber\\
 &= (\eps|A|)^{-2} \Bigg(\sum_{\substack{y,z\in A,\\ 
 \|y-z\|< \frac{n}{\ln^{M_0}n}}} \cov (\xi_y^{(k)},\xi_z^{(k)})
+ \sum_{\substack{y,z\in A,\\ 
 \|y-z\|\geq  \frac{n}{\ln^{M_0}n}}}
\cov (\xi_y^{(k)},\xi_z^{(k)})  \Bigg)
\nonumber\\
&\leq (\eps|A|)^{-2}
\Big(\sum_{y\in A} \big|A\cap 
\B(y,{\textstyle\frac{n}{\ln^{M_0}n}})\big| 
+ |A|^2 
 O\Big(\frac{\ln\ln n}{\ln n}\Big)
 \Big)\nonumber\\
&\leq \eps^{-2} 
\ell^{(n)}_A
+ \eps^{-2} 
 O\Big(\frac{\ln\ln n}{\ln n} \Big).
\label{bigcalc_Cheb}
\end{align}
 
Let 
\[
 \Phi^{(s)} = \min\big\{k: \V^{(k)}\leq s\big\}
\]
be the number of excursions necessary to make the 
unvisited proportion of~$A$ at most~$s$.
We have
\begin{align*}
 \IP[\V(A)\leq s] & = \IP[\Phi^{(s)}\leq N]\\
 &= \IP[\Phi^{(s)}\leq N,N=\hN ] + 
 \IP[\Phi^{(s)}\leq N,N\neq \hN ]\\
&= \IP[\Phi^{(s)}\leq \hN ]
+ \IP[\Phi^{(s)}\leq N,N\neq \hN ]
- \IP[\Phi^{(s)}\leq \hN ,N\neq \hN ],
\end{align*}
so, recalling~\eqref{N_neq_hatN},
\begin{equation}
\label{V_Phi_hatN}
 \big|\IP[\V(A)\leq s] - \IP[\Phi^{(s)}\leq \hN ]\big|
\leq \IP[N\neq \hN ] \leq O(n^{-1}).
\end{equation}

Next, we write
\begin{align}
\IP[\Phi^{(s)}\leq \hN ]
 &= \IE \big(\IP[\hN \geq \Phi^{(s)}\mid\Phi^{(s)}]\big)\nonumber \\
&=\IE (1-\psi_n)^{\Phi^{(s)}},
\label{genfunc}
\end{align}
(here we used the independence property stated 
below~\eqref{N_neq_hatN})
and concentrate on obtaining lower and upper bounds 
on the expectation in the right-hand side of~\eqref{genfunc}.
For this, assume that $s\in(0,1)$ is fixed and abbreviate 
\begin{align*}
 \delta_n &= \Big(\frac{\ln\ln n}{\ln n}\Big)^{1/3}\\
 k_n^{-} &= \Big\lfloor(1-\delta_n)\ln s^{-1} \frac{\ln n}{\ln\ln n}
   \Big\rfloor,\\
   k_n^{+} &= \Big\lceil(1+\delta_n)\ln s^{-1} \frac{\ln n}{\ln\ln n}
   \Big\rceil;
\end{align*}
we also assume that~$n$ is sufficiently large so that 
$\delta_n \in(0,\frac{1}{2})$ and $1<k_n^{-}<k_n^{+}$.
Now, according to~\eqref{est_mu_k},
\begin{align*}
 \mu_y^{(k_n^{\pm})} &= \exp\Big(
  -(1\pm \delta_n)\ln s^{-1}
\Big(1+O\Big((k_n^{\pm})^{-1}+\frac{\ln\ln n}{\ln n}\Big)\Big)\Big)\\
 & = s \exp\Big(
  -\ln s^{-1}
\Big(\pm \delta_n+O\Big((k_n^{\pm})^{-1}+\frac{\ln\ln n}{\ln n}\Big)\Big)\Big)\\
& = s \Big(1+O\Big(\delta_n\ln s^{-1}
+ \frac{\ln\ln n}{\ln n}(1+\ln s^{-1})\Big)\Big),
\end{align*}
so in both cases it holds that
(observe that $s\ln s^{-1}\leq 1/e$ for all $s\in[0,1]$)
\begin{equation}
\label{muk_ourcase} 
  \mu_y^{(k_n^{\pm})} 
= s + O\Big(\delta_n +\frac{\ln\ln n}{\ln n} \Big)
= s + O(\delta_n). 
\end{equation}
With a similar calculation, one can also observe that
\begin{equation}
\label{1-psi_ourcase} 
  (1-\psi_n)^{(k_n^{\pm})} 
= s + O(\delta_n). 
\end{equation}

We then write, using~\eqref{muk_ourcase} 
\begin{align}
 \IP[\Phi^{(s)} > k_n^{+}] &= \IP[\V^{(k_n^{+})}>s]\nonumber\\
&=  \IP\Big[|A|^{-1}\sum_{y\in A}\xi_y^{(k_n^{+})}> s\Big]\nonumber\\
 &= \IP\Big[|A|^{-1}\sum_{y\in A}(\xi_y^{(k_n^{+})}
   -\mu_y^{(k_n^{+})})> s 
   - |A|^{-1}\sum_{y\in A}\mu_y^{(k_n^{+})}\Big]\nonumber\\
 &= \IP\Big[|A|^{-1}\sum_{y\in A}(\xi_y^{(k_n^{+})}
   -\mu_y^{(k_n^{+})})>
O(\delta_n)
\Big].
\label{Cheb+}
\end{align}
Then, \eqref{bigcalc_Cheb} implies that 
\begin{equation}
\label{est_k+}
   \IP[\Phi^{(s)} > k_n^{+}] 
\leq O\Big(\ell^{(n)}_A\Big(\frac{\ln\ln n}{\ln n}\Big)^{-2/3}
 + \Big(\frac{\ln\ln n}{\ln n}\Big)^{1/3} \Big).
\end{equation}
Quite analogously, one can also obtain that
\begin{equation}
\label{est_k-}
   \IP[\Phi^{(s)} < k_n^{-}] 
\leq O\Big(\ell^{(n)}_A\Big(\frac{\ln\ln n}{\ln n}\Big)^{-2/3}
 + \Big(\frac{\ln\ln n}{\ln n}\Big)^{1/3} \Big).
\end{equation}
Using~\eqref{1-psi_ourcase} and~\eqref{est_k+}, we then write
\begin{align}
 \IE (1-\psi_n)^{\Phi^{(s)}} &\geq
   \IE \big((1-\psi_n)^{\Phi^{(s)}}\1{\Phi^{(s)} \leq k_n^{+}}\big)
\nonumber\\
 &\geq (1-\psi_n)^{k_n^{+}}
   \IP[\Phi^{(s)} \leq k_n^{+}]\nonumber\\
 &\geq \Big(s-O\Big(\Big(\frac{\ln\ln n}{\ln n}\Big)^{1/3}\Big)\Big)
\Big(1-O\Big(\ell^{(n)}_A\Big(\frac{\ln\ln n}{\ln n}\Big)^{-2/3}
 + \Big(\frac{\ln\ln n}{\ln n}\Big)^{1/3} \Big)\Big),
\label{bigest>}
\end{align}
and, using~\eqref{1-psi_ourcase} and~\eqref{est_k-},
\begin{align}
 \IE (1-\psi_n)^{\Phi^{(s)}} &=
   \IE \big((1-\psi_n)^{\Phi^{(s)}}\1{\Phi^{(s)} \geq k_n^{-}}\big)
 + \IE \big((1-\psi_n)^{\Phi^{(s)}}\1{\Phi^{(s)} < k_n^{-}}\big)
\nonumber\\
 &\leq (1-\psi_n)^{k_n^{-}}
   + \IP[\Phi^{(s)} < k_n^{-}]
\nonumber\\
 &\leq \Big(s+O\Big(\Big(\frac{\ln\ln n}{\ln n}\Big)^{1/3}\Big)\Big)
\nonumber\\
& \qquad
+
\Big(1-O\Big(\ell^{(n)}_A\Big(\frac{\ln\ln n}{\ln n}\Big)^{-2/3}
 + \Big(\frac{\ln\ln n}{\ln n}\Big)^{1/3} \Big)\Big).
\label{bigest<}
\end{align}
%
%
%
%
Therefore, using also~\eqref{V_Phi_hatN}--\eqref{genfunc},
we obtain~\eqref{main_res}, thus concluding the 
proof of Theorem~\ref{t_main_res}.
\end{proof}

Next, we will prove Theorems~\ref{t_rec_trans}
and~\ref{t_Liouville}, since the latter will be needed
in the course of the proof of Theorem~\ref{t_bigholes}. 

\begin{proof}[Proof of Theorem~\ref{t_rec_trans}]
Clearly, we only need to prove that every infinite subset
of~$\Z^d$ is recurrent for~$\s$.  
Basically, this is a consequence of the fact that,
due to~\eqref{not_hit_site},
\begin{equation}
\label{lim1/2}
 \lim_{y\to\infty} \IP_{x_0}\big[\htau_1(y)<\infty\big] = \frac{1}{2}
\end{equation}
for any $x_0\in\Z^2$. Indeed, let~$\s_0=x_0$; since~$A$ is infinite,
by~\eqref{lim1/2} one can find~$y_0\in A$ and~$R_0$ such that 
$\{x_0,y_0\}\subset \B(R_0)$ and
\[
 \IP_{x_0}\big[\htau_1(y_0)<\htau_1(R_0)\big] \geq \frac{1}{3}.
\]
Then, for any~$x_1\in \partial\B(R_0)$, we can find~$y_1\in A$ 
and~$R_1>R_0$ such that 
$y_1\in \B(R_1)\setminus \B(R_0)$ and
\[
 \IP_{x_1}\big[\htau_1(y_1)<\htau_1(R_1)\big] \geq \frac{1}{3}.
\]
Continuing in this way, we can construct a sequence
$R_0<R_1<R_2<\ldots$ (depending on the set~$A$) such that,
for each $k\geq 0$, the walk~$\s$ hits~$A$ 
on its way from~$\partial\B(R_k)$ to~$\partial\B(R_{k+1})$
with probability at least~$\frac{1}{3}$,
regardless of the past. 
This clearly implies that~$A$ is a recurrent set.
\end{proof}

\begin{proof}[Proof of Theorem~\ref{t_Liouville}]
Indeed, Theorem~\ref{t_rec_trans} implies that 
every subset of~$\Z^2$ must be either recurrent or transient,
and then Proposition~3.8 in Chapter 2 of~\cite{R84}
implies the Liouville property. Still, for the reader's
convenience, we include the proof here.
 Assume that $h:\Z^2\setminus\{0\} \to \R$ is a bounded
harmonic function for~$\s$. Let us prove that 
\begin{equation}
\label{limit_infty}
 \liminf_{y\to \infty}h(y) = \limsup_{y\to \infty}h(y),
\end{equation}
that is, $h$ must have a limit at infinity. Indeed, 
assume that~\eqref{limit_infty} does not hold, which means that
there exist two constants~$b_1<b_2$ and
two \emph{infinite} sets $B_1,B_2\subset\Z^2$
such that $h(y)\leq b_1$ for all $y\in B_1$
and $h(y) \geq b_2$ for all $y\in B_2$.
Now, on one hand~$h(\s_n)$ is a bounded martingale, so
it must a.s.\ converge to some limit; on the other hand,
Theorem~\ref{t_rec_trans} implies that both~$B_1$ and~$B_2$
will be visited infinitely often by~$\s$, and so~$h(\s_n)$ 
cannot converge to any limit, thus yielding a contradiction.
This proves~\eqref{limit_infty}.

Now, if $\lim_{y\to \infty}h(y)=c$, then it is easy to 
obtain from the Maximum Principle that $h(x)=c$ for any~$x$.
This concludes the proof of Theorem~\ref{t_Liouville}.
\end{proof}

Finally, we are able to prove that there are ``big holes'' in the 
range of~$\s$:
\begin{proof}[Proof of Theorem~\ref{t_bigholes}]
Clearly, if~$G$ does not surround the origin in the sense of 
Definition~\ref{df_surround}, then $G\subset \B(c_1)\setminus \B(c_3)$.
For the sake of simplicity,
let us assume that $G\subset\B(1)\setminus \B(1/2)$;
the general case can be treated in a completely analogous way.

Consider the two sequences of events
\begin{align*}
 E_n &= \big\{\htau_1(2^{3n-1}G)>\htau_1(2^{3n}),
    \|\s_j\| > 2^{3n-1} \text{ for all }
       j\geq \htau_1(2^{3n})\big\},\\
   E'_n &= \big\{
    \|\s_j\| > 2^{3n-1} \text{ for all }
       j\geq \htau_1(2^{3n})\big\}    
\end{align*}
and note that $E_n\subset E'_n$ and
$2^{3n-1}G\cap \s_{[0,\infty)}=\emptyset$ on~$E_n$.
Our goal is to show that a.s.\ an infinite number of  
events $(E_n, n\geq 1)$ occurs. 
Observe, however, that the events in each of the 
above two sequences are \emph{not} independent,
so the ``basic'' second Borel-Cantelli lemma will not work.

 In the following, we use a generalization of the second 
Borel-Cantelli lemma, known as 
the Kochen-Stone theorem~\cite{KS64}:
it holds that
\begin{equation}
\label{Kochen-Stone}
 \IP\Big[\sum_{k=1}^\infty \1{E_k}=\infty\Big] \geq
  \limsup_{k\to\infty} \frac{\big(\sum_{i=1}^k\IP[E_i]\big)^2}
  {\sum_{i,j=1}^k\IP[E_i\cap E_j]}.
\end{equation}

We will now prove that there exists a positive constant~$c_4$
such that
\begin{equation}
\label{lower_En}
 \IP[E_n] \geq \frac{c_4}{n} \quad \text{ for all }n\geq 1.
\end{equation}
Indeed, since $G\subset\B(1)\setminus \B(1/2)$ does not
surround the origin, by comparison with Brownian motion
it is elementary to obtain that, for some~$c_5>0$,
\[
 \IP_x\big[\tau_1(2^{3n-1}G)>\tau_1(2^{3n}),
   \tau_1(0)>\tau_1(2^{3n}) \big] > c_5
\]
for all~$x\in\partial \B(2^{3(n-1)})$.
Lemma~\ref{l_relation_S_hatS} then implies that,
for some~$c_6>0$,
\begin{align}
\lefteqn{
 \IP_x\big[\htau_1(2^{3n-1}G)>\htau_1(2^{3n})\big] 
}\nonumber\\
&= \big(1+o(2^{-3n})\big) 
  \IP_x\big[\tau_1(2^{3n-1}G)>\tau_1(2^{3n})
   \mid \tau_1(0)>\tau_1(2^{3n})\big] \nonumber\\
&= \big(1+o(2^{-3n})\big) 
  \IP_x\big[\tau_1(2^{3n-1}G)>\tau_1(2^{3n}),
  \tau_1(0)>\tau_1(2^{3n})\big] 
 > c_6
\label{inv_pri>c}
\end{align}
for all~$x\in\partial \B(2^{3(n-1)})$.
Let us denote, recalling~\eqref{formula_for_a},
$\gamma^{*}=\frac{\pi}{2}\times\frac{1}{\ln 2} \times
\frac{2\gamma+ 3\ln 2}{\pi}
=\frac{2\gamma+ 3\ln 2}{2\ln 2}$.
Using~\eqref{escape_condS}, we then obtain
\begin{align}
 \IP_z\big[ \|\s_j\| > 2^{3n-1} \text{ for all }
       j\geq 0\big] 
& = 1- \frac{a(2^{3n-1})+O(2^{-3n})}{a(2^{3n})+O(2^{-3n})}
\nonumber\\
&= \frac{1}{3n+\gamma^{*}}\big(1+o(2^{-3n})\big).
\label{>c/n}
\end{align}
for any $z\in\partial \B(2^{3n})$.
The inequality~\eqref{lower_En} follows from~\eqref{inv_pri>c}
and~\eqref{>c/n}.

Now, we need an upper bound for $\IP[E_m\cap E_n]$,
$m\leq n$. Clearly, $E_m \cap E_n\subset E'_m \cap E'_n$,
and note that the event~$E'_m \cap E'_n$ means that 
the particle hits $\partial\B(2^{3n})$ before
$\partial\B(2^{3m-1})$ starting from a site on~$\partial\B(2^{3m})$,
and then never hits $\partial\B(2^{3n-1})$ starting from a site
 on~$\partial\B(2^{3n})$.
So, again using~\eqref{escape_condS} and 
Lemma~\ref{l_relation_S_hatS}, we write
analogously to~\eqref{>c/n}
(and also omitting a couple of lines of elementary calculations)
\begin{align}
 \IP[E_m \cap E_n] & \leq \IP[E'_m \cap E'_n]\nonumber\\
 & = \frac{(a(2^{3m-1}))^{-1}-(a(2^{3m}))^{-1}+O(2^{-3m})}
{(a(2^{3m-1}))^{-1}-(a(2^{3n}))^{-1}+O(2^{-3m})} 
\times \Big(1- \frac{a(2^{3n-1})+O(2^{-3n})}{a(2^{3n})+O(2^{-3n})}\Big)
\nonumber\\
 & = \frac{1}{(3(n-m)+1)(3m+\gamma^{*})}\big(1+o(2^{-3m})\big).
\label{EmEn_upper}
\end{align}

Now, \eqref{lower_En} implies that   
$\sum_{i=1}^k\IP[E_i]\geq c_{9}\ln k$,
and~\eqref{EmEn_upper} implies (again, 
after some elementary calculations) that 
$\sum_{i,j=1}^k\IP[E_i\cap E_j]\leq c_{10}\ln^2 k$.
So, using~\eqref{Kochen-Stone},
we obtain that 
\[
 \IP\Big[\sum_{k=1}^\infty \1{E_k}=\infty\Big] \geq c_{11}>0.
\]
Now, note that, 
again due to Proposition 3.8 in Chapter 2 of~\cite{R84}, 
the Liouville property implies that every tail event
must have probability~$0$ or~$1$, and so the probability
in the above display must be equal to~$1$.
This concludes the proof 
of Theorem~\ref{t_bigholes}.
\end{proof}

\section*{Acknowledgements}
The work of S.P.\ and M.V.\ was partially supported by
CNPq (grants 300886/2008--0 and 305369/2016--4) 
and FAPESP (grant 2017/02022--2). 
The authors are grateful to the anonymous referee 
for carefully reading the first version of this paper.

\end{document}